\documentclass{amsart} % in was article
\usepackage{amssymb}
\usepackage{amsmath}
\usepackage{amsfonts}

\begin{document}
%%%%%%%%%%%%%%%%%%%%%%%%%%%%%%%%%%%%%%%%%%%%%%%%%%%%%%%%%%%%%%%%%%%%%%
%	spaces for your own definitions follows
%%%%%%%%%%%%%%%%%%%%%%%%%%%%%%%%%%%%%%%%%%%%%%%%%%%%%%%%%%%%%%%%%%%%%%
\newtheorem{theo}{Theorem}[section]
\newtheorem{atheo}{Theorem*}
\newtheorem{prop}[theo]{Proposition}
\newtheorem{aprop}[atheo]{Proposition*}
\newtheorem{lemma}[theo]{Lemma}
\newtheorem{alemma}[atheo]{Lemma*}
\newtheorem{exam}[theo]{Example}
\newtheorem{coro}[theo]{Corollary}
\theoremstyle{definition}
\newtheorem{defi}[theo]{Definition}
\newtheorem{rem}[theo]{Remark}

%\renewcommand{\theequation}{\mbox{\arabic{section}.\arabic{equation}}}

%letters - added these
\newcommand{\Bb}{{\bf B}}
\newcommand{\Cb}{{\mathbb C}}
\newcommand{\Nb}{{\mathbb N}}
\newcommand{\Qb}{{\mathbb Q}}
\newcommand{\Rb}{{\mathbb R}}
\newcommand{\Zb}{{\mathbb Z}}
\newcommand{\Ac}{{\mathcal A}}
\newcommand{\Bc}{{\mathcal B}}
\newcommand{\Cc}{{\mathcal C}}
\newcommand{\Dc}{{\mathcal D}}
\newcommand{\Fc}{{\mathcal F}}
\newcommand{\Ic}{{\mathcal I}}
\newcommand{\Jc}{{\mathcal J}}
\newcommand{\Kc}{{\mathcal K}}
\newcommand{\Lc}{{\mathcal L}}
\newcommand{\Oc}{{\mathcal O}}
\newcommand{\Pc}{{\mathcal P}}
\newcommand{\Sc}{{\mathcal S}}
\newcommand{\Tc}{{\mathcal T}}
\newcommand{\Uc}{{\mathcal U}}
\newcommand{\Vc}{{\mathcal V}}

\author{Charles A. Akemann and Nik Weaver}

\title [Detecting Fourier subspaces]
       {Detecting Fourier subspaces}

\address {Department of Mathematics\\
UCSB\\
Santa Barbara, CA 93106\\
Department of Mathematics\\
Washington University\\
Saint Louis, MO 63130}

\email {akemann@math.ucsb.edu, nweaver@math.wustl.edu}

\date{\em March 23, 2015}

\thanks{Second author partially supported by NSF grant DMS-1067726.}

%%%%%%%%%%%%%%%%%%%%%%%%%%%%%%%%%%%%%%%%%%%%%%%%%%%%%%%%%%%%%%%%%%%%%%%
%	Please insert the article body now
%%%%%%%%%%%%%%%%%%%%%%%%%%%%%%%%%%%%%%%%%%%%%%%%%%%%%%%%%%%%%%%%%%%%%%%

\begin{abstract}
Let $G$ be a finite abelian group. We examine the discrepancy between
subspaces of $l^2(G)$ which are diagonalized in the standard basis and
subspaces which are diagonalized in the dual Fourier basis. The general
principle is that a Fourier subspace whose dimension is small compared
to $|G| = {\rm dim}(l^2(G))$ tends to be far away from standard subspaces.
In particular,
the recent positive solution of the Kadison-Singer problem shows that from
within any Fourier subspace whose dimension is small compared to $|G|$
there is standard subspace which is essentially indistinguishable from
its orthogonal complement.
\end{abstract}

\maketitle

The purpose of this note is to describe a simple application of the recent
solution of the Kadison-Singer problem \cite{mss1, mss2} to a question in
harmonic analysis and signal analysis.

Let $G$ be a finite abelian group equipped with counting measure, and
for each $g \in G$ let $e_g \in l^2(G)$ be the function which takes the
value $1$ at $g$ and is zero elsewhere. Then $\{e_g: g \in G\}$ is an
orthonormal basis of $l^2(G)$; call it the {\it standard} basis.

Another nice basis of $l^2(G)$ comes from the dual group $\hat{G}$,
the set of characters of $G$, i.e., homomorphisms from $G$ into the
circle group ${\mathbb T}$. Every character has $l^2$ norm equal to
$|G|^{1/2}$, where $|G|$ is the cardinality of $G$, and the normalized
set $\{\hat{e}_\phi = |G|^{-1/2}\phi: \phi \in \hat{G}\}$ is also
an orthonormal basis of $l^2(G)$. We call this the {\it Fourier} basis.  (Note that when every element of $G$ has order 2, then the Fourier basis forms the rows of a Hadamard matrix \cite{YNX}.  The method of this paper applies to bases of this type as well even if they don't arise from a Fourier transform.)

Say that a subspace of $l^2(G)$ is {\it standard} if it is the span of
some subset of the standard basis, and {\it Fourier} if it is the span
of some subset of the Fourier basis. Now
each $\hat{e}_\phi$ is as far away from the standard basis as possible
in the sense that $|\langle \hat{e}_\phi, e_g\rangle| = |G|^{-1/2}$
for all $g \in G$ and $\phi \in \hat{G}$. However, Fourier subspaces can
certainly intersect standard subspaces --- trivially, $l^2(G)$ is itself
both a standard subspace and a Fourier subspace.

A more interesting question is whether Fourier subspaces whose dimensions
are ``small'' compared to $|G|$ can intersect standard subspaces which are
small in the same sense. This could be of interest in relation to signal
analysis, say, if we are trying to detect a signal by measuring a
relatively small number of frequencies.

(By interchanging the roles of a group and its dual group, we see that
the problem of detecting standard subspaces using Fourier subspaces is
equivalent to the problem of detecting Fourier subspaces using standard
subspaces. But in keeping with the signal analysis perspective, we will
stick with the first formulation.)

The basic obstacle to having small standard and Fourier subspaces
which intersect is the uncertainty principle for finite abelian groups
\cite{DS, TT, terras}. According to this principle, if a nonzero function
$f \in l^2(G)$ is supported on a set $S \subseteq G$
and its Fourier transform is supported on $T \subseteq \hat{G}$ ---
meaning that $\langle f, \hat{e}_\phi\rangle \neq 0$ only for $\phi \in T$
--- then $|S|\cdot |T| \geq |G|$. In the special case where $G$ has prime
order $p$, we have the much stronger inequality $|S| + |T| \ge p+1$, and
this inequality is absolutely sharp  \cite{TT}.
Intuitively, if $f$ is very localized with
respect to the standard basis then it must be ``spread out'' with respect
to the Fourier basis. In terms of subspaces, the result can be stated
as follows:

\begin{prop}
Let $G$ be a finite abelian group and let $E$ and $F$ respectively be
standard and Fourier subspaces of $l^2(G)$. If ${\rm dim}(E)\cdot{\rm dim}(F)
< |G|$ then $E \cap F = \{0\}$.  If $|G|$ is prime and
${\rm dim}(E) + {\rm dim}(F) \le |G|$ then $E \cap F = \{0\}$.
\end{prop}

This follows immediately from the uncertainty principles described above
because the dimensions of $E$ and $F$ equal the number of basis elements
which span them, so that the support of any element of $E$ (respectively,
$F$) has cardinality at most ${\rm dim}(E)$ (respectively, ${\rm dim}(F)$).  
(Uncertainty principles are related to signal reconstruction in a
different way in \cite{crt} and related papers.)

So standard and Fourier subspaces must intersect only in $\{0\}$ if both are
sufficiently small. However, according to the multiplicative bound in the
last proposition,  both dimensions
could be as small as $|G|^{1/2}$, which is small compared to $|G|$ when
$|G|$ is large. This means that the multiplicative bound does not prevent
standard and Fourier subspaces whose dimensions are small compared to $|G|$
from intersecting, although the additive bound when $|G|$ is prime certainly
does. Indeed, there are easy examples of intersecting standard and Fourier
subspaces whose dimensions are both equal to $|G|^{1/2}$. 

\begin{exam}
Let $G = {\mathbb Z}/n^2{\mathbb Z}$, where $|G| = n^2$ and the characters
have the form $\phi_b: a \mapsto e^{2\pi i ab/n^2}$ for $a, b \in
{\mathbb Z}/n^2{\mathbb Z}$. Here the function $f = \sum_{b=0}^{n-1} \phi_{nb}
\in l^2(G)$ satisfies
$$f(a) = \sum_{b=0}^{n-1} e^{2\pi i ab/n} =
\begin{cases}
n&\mbox{if $a \equiv 0$ (mod $n$)}\cr
0&\mbox{if $a \not\equiv 0$ (mod $n$)},
\end{cases}$$
so that $f$ belongs both to an $n$-dimensional Fourier subspace (directly
from its definition) and to an $n$-dimensional standard subspace (by the
preceding calculation).
\end{exam}

From the point of view of signal analysis, however, we are probably not so
interested in intersecting a single standard subspace. If we do not know
where the signal we want to detect is supported, we would presumably want to
intersect, if not every standard subspace, at least every standard subspace
whose dimension is greater than some threshold value. But unless the
relevant dimensions are large, this is impossible for elementary linear
algebra reasons. We have the following simple result:

\begin{prop}
Let $G$ be a finite abelian group and let $F$ be a
Fourier subspace of $l^2(G)$. Then there exists a standard subspace $E$ with
${\rm dim}(E) = |G| - {\rm dim}(F)$ such that $E \cap F = \{0\}$.
\end{prop}

The proof is easy. Starting with the standard basis $B = \{e_g: g \in G\}$
of $l^2(G)$ and the basis $B' = \{\hat{e}_\phi: \phi \in T\}$ of $F$, we can
successively replace distinct elements of $B$ with elements of $B'$ in such a way
 that the set remains linearly independent at each step. The end result
will be a
(non orthogonal) basis of $l^2(G)$ of the form $B' \cup \{e_g: g \in S\}$
for some set $S \subseteq G$ with $|S| = |G| - {\rm dim}(F)$. Then
$E = {\rm span}\{e_g: g \in S\}$ is the desired standard subspace which
does not intersect $F$.

Thus, if the dimension of a Fourier subspace is small compared to $|G|$,
there will be many standard subspaces whose dimensions are small compared
to $|G|$ which it does not intersect. But requiring subspaces to
intersect is a very strong condition. Merely asking that the Fourier
subspace contain unit vectors which are close to the standard subspace
makes the problem much more interesting and difficult. Nonetheless, we
can still get something from simply counting dimensions. Namely, if
$\hat{e}_\phi$ is any Fourier basis vector and $E$ is a standard
subspace, then, letting $P$ be the orthogonal projection onto
$E = {\rm span}\{e_g: g \in S\}$, we have
$$\|P\hat{e}_\phi\|^2
= \left\|\sum_{g \in S} \langle \hat{e}_\phi, e_g\rangle e_g\right\|^2
= \frac{{\rm dim}(E)}{|G|}.$$
Thus, even a single Fourier basis vector can ``detect'' arbitrary standard
subspaces to the extent that those subspaces have dimension comparable to
$|G|$.

Can Fourier subspaces do better? The relevant gauge here is the quantity
$$\|PQ\| = \|QP\| = \sup \{\|Qv\|: v \in E, \|v\| = 1\}$$
where $P$ and $Q$ are the orthogonal projections onto a standard subspace
$E$ and a Fourier subspace $F$, respectively. It effectively measures the
minimal angle between $E$ and $F$. The surprisingly strong result is
that, by this measure, so long as the dimension of a Fourier subspace is
small compared to $|G|$, there are standard subspaces which it detects
only marginally better than a single Fourier basis vector does.

One might suspect that a randomly chosen standard subspace would demonstrate
that claim. Maybe this technique
would work for most Fourier subspaces, but it does not in general,
even for Fourier subspaces whose dimension is small compared to $|G|$. 

\begin{exam}
Recall Example 0.2
where an $n$-dimensional Fourier subspace $F$ intersected an $n$-dimensional
standard subspace $E$. Here ${\rm dim}(F)/|G| = 1/n$, which can be as small
as we like. Now consider the group $G' = {\mathbb Z}/n^2{\mathbb Z} \times
{\mathbb Z}/N{\mathbb Z}$ where $N$ is large compared to $n$. We have a
natural identification
$$l^2(G') \cong
l^2({\mathbb Z}/n^2{\mathbb Z}) \otimes l^2({\mathbb Z}/N{\mathbb Z}),$$
under which identification $F\otimes l^2({\mathbb Z}/N{\mathbb Z})$ is
a Fourier subspace. The ratio of its dimension to the cardinality of $G'$
is still $1/n$. But for sufficiently large $N$, a randomly chosen subset of
$G'$ will contain at least one element of $S \times {\mathbb Z}/N{\mathbb Z}$
with high probability, where $E = {\rm span}\{e_g: g \in S\}$. One can even
say this of a randomly chosen subset of $G'$ of cardinality $|G'|/n$.
This means that a randomly chosen standard subspace of dimension $|G'|/n$
will intersect the Fourier subspace $F\otimes l^2({\mathbb Z}/N{\mathbb Z})$
with high probability --- depending on the value of $N$, with probability as
close to 1 as we like.

\end{exam}

Nonetheless, if a Fourier subspace $F$ has relatively small dimension, there
will always exist standard subspaces of arbitrary dimension which are barely
closer to $F$ than they are to a single Fourier basis vector. We have the
following theorem:

\begin{theo}\label{onesided}
Let $G$ be a finite abelian group, let $F$ be a Fourier subspace of
$l^2(G)$, and let $Q$ be the orthogonal projection onto $F$. Then for
any $k \leq |G|$ there is a set $S \subseteq G$ with $|S| = k$ and
such that
$$\|PQ\|^2 \leq \frac{k}{|G|} + O(\sqrt{\epsilon}),$$
where $P$ is the orthogonal projection onto
${\rm span}\{e_g: g \in S\}$ and $\epsilon = {\rm dim}(F)/|G|$.
\end{theo}

This result is strengthened further by Theorem \ref{twosided} below, but Theorem 0.5
has an easier proof which is of independent interest. We include this
proof in the appendix.

The strengthened version of Theorem \ref{onesided} simultaneously asks the
same question about the complementary standard subspace. If our goal is
detection, then the worst that could happen here is that a Fourier subspace
$F$ does essentially no better than a single Fourier basis vector at detecting
either some standard subspace $E$ or its orthocomplement $E^\perp$. In fact,
this worst-case scenario is realized: we can show that every Fourier subspace
of small dimension relative to $|G|$ fails to do significantly better than
a single Fourier basis vector at detecting both a sequence of
standard subspaces of varying dimension and their orthocomplements.

In this case, constructing the undetectable standard subspaces is no longer
merely a matter of controlling the
largest eigenvalue of $QPQ$ (which suffices because $\|QPQ\| = \|PQ\|^2$).
Now we also have to control the largest eigenvalue of $Q(1-P)Q$, and
this is a Kadison-Singer type setup. Although the problem we consider here
is not as general as the full Kadison-Singer problem, the core difficulty
is clearly already present. Thus, one should not expect any
easier proof than the ones that appear in \cite{mss1}, \cite{mss2} or the remarkable generalization in  \cite{PB}.

\begin{theo}\label{twosided}
Let $G$ be a finite abelian group, let $F$ be a Fourier subspace of
$l^2(G)$, and let $Q$ be the orthogonal projection onto $F$. Then for
any $k \leq |G|$ there is a set $S \subseteq G$ with $|S| = k$ and
such that both
$$\|PQ\|^2 \leq \frac{k}{|G|} + O(\sqrt{\epsilon})$$
and
$$\|(I-P)Q\|^2 \leq \frac{|G| - k}{|G|} + O(\sqrt{\epsilon})$$
where $P$ is the orthogonal projection onto
${\rm span}\{e_g: g \in S\}$ and $\epsilon = {\rm dim}(F)/|G|$.
\end{theo}

\begin{proof}
For each $g \in G$ let $u_g = Qe_g \in F$. Then the
rank one operators $u_gu_g^*: f \mapsto \langle f, u_g\rangle u_g$ satisfy
$$\sum_{g\in G} u_gu_g^* f = \sum_{g\in G} \langle f, Qe_g\rangle\cdot Qe_g
= Q\left(\sum_{g\in G} \langle Qf, e_g\rangle e_g\right) = Qf$$
for all $f \in l^2(G)$; that is, $\sum_{g \in G} u_gu_g^* = Q$.
We also have $\|u_g\|^2 = \|Qe_g\|^2 = \epsilon$ for
all $g$. So by (\cite{akewea}, comment following Corollary 1.2), for any
$k \leq |G|$ there is a set $S \subseteq G$ such that
$$\left\|\sum_{g \in S} u_gu_g^* - \frac{k}{|G|}Q\right\|
= O(\sqrt{\epsilon}).$$
Letting $P$ be the orthogonal projection onto ${\rm span}\{e_g: g \in S\}$,
we have $P = \sum_{g \in S} e_ge_g^*$, and it follows that
$$\left\|QPQ - \frac{k}{|G|}Q\right\|
= \left\|\sum_{g \in S} u_gu_g^* - \frac{k}{|G|}Q\right\|
= O(\sqrt{\epsilon}),$$
which also implies
$$\left\|Q(I - P)Q - \frac{|G| - k}{|G|}Q\right\|
= O(\sqrt{\epsilon}).$$
We therefore have
$$\|PQ\|^2 = \|QPQ\| \leq \frac{k}{|G|} + O(\sqrt{\epsilon})$$
and
$$\|(I-P)Q\|^2 = \|Q(I-P)Q\| \leq \frac{|G| - k}{|G|} + O(\sqrt{\epsilon}),$$
as desired.
The set $S$ might not have cardinality exactly $k$, but it cannot have
cardinality greater than $k + O(\sqrt{\epsilon})$ or less than
$k - O(\sqrt{\epsilon})$, so it can be adjusted to have cardinality $k$
without affecting the order of the estimate, if needed.
\end{proof}

In particular, taking $k \approx |G|/2$ in Theorem \ref{twosided}, we get
$$\|PQ\| \approx \|(1-P)Q\| \approx \frac{1}{\sqrt{2}}.$$
This implies that every nonzero vector in ${\rm ran}(P)$ is roughly $45^\circ$
away from $F$, and the same is true of every nonzero vector in
${\rm ran}(P)^\perp$. That is, if $\epsilon$ is small then from within $F$
the two subspaces are essentially indistinguishable. Another way to say
this is that every vector in $F$ has roughly half of its $l^2$ norm
supported on $S$ and roughly half supported on $G\setminus S$.

\appendix
\section{}

We prove Theorem \ref{onesided}. The argument is a straightforward
application of the spectral sparsification technique introduced in
Srivastava's thesis \cite{S}.

Let $\{u_1, \ldots, u_n\}$ be a finite set of
vectors in ${\bf C}^m$, each of norm $\sqrt{\epsilon}$, satisfying
\[\sum_{i=1}^n u_iu_i^* = I\]
where $u_iu_i^*$ is the rank one operator on ${\bf C}^m$ defined by
$u_iu_i^*: v \mapsto \langle v,u_i\rangle u_i$ and $I$ is the identity
operator on ${\bf C}^m$. Note that
${\rm Tr}(u_iu_i^*) = {\rm Tr}(u_i^*u_i)
= \|u_i\|^2 = \epsilon$; since ${\rm Tr}(I) = m$,
it follows that $n\epsilon = m$.

Let $k < n$. As in \cite{S}, we will build a sequence
$u_{i_1}, \ldots, u_{i_k}$ one element at a time. The construction is
controlled by the behavior of the operators
$A_j = \sum_{d=1}^j u_{i_d}u_{i_d}^*$ using the following tool.
For any positive operator $A$ and any $a > \|A\|$, define the
{\it upper potential} $\Phi^a(A)$ to be
\[\Phi^a(A) = {\rm Tr}((aI - A)^{-1});\]
then, having chosen the vectors $u_{i_1}, \ldots, u_{i_{j-1}}$, the plan
will be to select a new vector $u_{i_j}$ so as to minimize $\Phi^{a_j}(A_j)$,
where the $a_j$ are an increasing sequence of upper bounds.
This potential function disproportionately penalizes eigenvalues which are
close to $a_j$ and thereby controls the maximum eigenvalue, i.e., the norm,
of $A_j$. The key fact about the upper potential is given in the following
result.

\begin{lemma}\label{lem1}
(\cite{S}, Lemma 3.4)
Let $A$ be a positive operator on ${\bf C}^m$, let $a, \delta > 0$, and
let $v \in {\bf C}^m$. Suppose $\|A\| < a$. If
\[\frac{\langle ((a + \delta)I - A)^{-2}v,v\rangle}{\Phi^a(A)
- \Phi^{a+\delta}(A)}
+ \langle ((a + \delta)I - A)^{-1}v,v\rangle \leq 1\]
then $\|A + vv^*\| < a + \delta$ and
$\Phi^{a + \delta}(A + vv^*) \leq \Phi^a(A)$.
\end{lemma}

The proof relies on the Sherman-Morrison formula, which states that if $A$
is positive and invertible then $(A + vv^*)^{-1} = A^{-1} -
\frac{A^{-1}(vv^*)A^{-1}}{I + \langle A^{-1}v,v\rangle}$.

We also require a simple inequality.

\begin{lemma}\label{lem2}
Let $a_1 \leq \cdots \leq a_m$ and $b_1 \geq \cdots \geq b_m$ be sequences
of positive real numbers, respectively increasing and decreasing. Then
$\sum a_ib_i \leq \frac{1}{m}\sum a_i \sum b_i$.
\end{lemma}

\begin{proof}
Let $M = \frac{1}{m}\sum b_i$. We want to show that $\sum a_i b_i \leq
\sum a_iM$, i.e., that $\sum a_i(b_i - M) \leq 0$. Since the sequence $(b_i)$
is decreasing, we can find $j$ such that $b_i \geq M$ for $i \leq j$ and
$b_i < M$ for $i > j$. Then $\sum_{i=1}^j a_i(b_i - M) \leq a_j\sum_{i=1}^j
(b_i-M)$ (since the $a_i$ are increasing and the values $b_i - M$ are positive)
and $\sum_{i=j+1}^m a_i(b_i-M) \leq a_j\sum_{i=j+1}^m (b_i-M)$
(since the $a_i$ are increasing and the values $b_i - M$ are negative). So
\[\sum_{i=1}^m a_i(b_i-M) \leq a_j \sum_{i=1}^m (b_i-M) = 0,\]
as desired.
\end{proof}

\begin{theo}\label{thm1}
Let $m \in {\bf N}$ and $\epsilon > 0$, and suppose $\{u_1, \ldots, u_n\}$ is
a finite sequence of vectors in ${\bf C}^m$ satisfying $\|u_i\|^2 = \epsilon$
for $1 \leq i \leq n$ and $\sum_{i=1}^n u_iu_i^* = I$. Then for any $k \leq n$
there is a set $S \subseteq \{1, \ldots, n\}$ with $|S| = k$ such that
\[\left\|\sum_{i \in S} u_iu_i^*\right\| \leq \frac{k}{n} +
O(\sqrt{\epsilon}).\]
\end{theo}

\begin{proof}
Define $a_j = \sqrt{\epsilon} +
\frac{1}{1-\sqrt{\epsilon}}\cdot\frac{j}{n}$ for $0 \leq j \leq k$.
We will find a sequence of distinct indices $i_1, \ldots, i_k$ such that
the operators $A_j = \sum_{d = 1}^j u_{i_d}u_{i_d}^*$,
$0 \leq j \leq k$, satisfy $\|A_j\| < a_j$ and $\Phi^{a_0}(A_0) \geq
\cdots \geq \Phi^{a_k}(A_k)$. Thus
\[\|A_k\| < \sqrt{\epsilon}
+ \frac{1}{1-\sqrt{\epsilon}}\cdot\frac{k}{n}
= \frac{k}{n} + O(\sqrt{\epsilon}),\]
yielding the desired conclusion. We start with $A_0 = 0$, so that
$\Phi^{a_0}(A_0) = \Phi^{\sqrt{\epsilon}}(0) =
{\rm Tr}((\sqrt{\epsilon}I)^{-1}) = m/\sqrt{\epsilon}$.

To carry out the induction step, suppose $u_{i_1}, \ldots, u_{i_j}$ have
been chosen. Let $\lambda_1 \leq \cdots \leq \lambda_m$ be the eigenvalues
of $A_j$. Then the eigenvalues of $I - A_j$ are $1-\lambda_1 \geq \cdots
\geq 1 - \lambda_m$ and the eigenvalues of $(a_{j+1}I - A_j)^{-1}$ are
$\frac{1}{a_{j+1} - \lambda_1} \leq \cdots \leq \frac{1}{a_{j+1} - \lambda_m}$.
Thus by Lemma \ref{lem2}
\begin{eqnarray*}
{\rm Tr}((a_{j+1}I - A_j)^{-1}(I - A_j))
&=& \sum_{i=1}^m \frac{1}{a_{j+1} - \lambda_i}(1 - \lambda_i)\cr
&\leq& \frac{1}{m}\sum_{i=1}^m \frac{1}{a_{j+1} - \lambda_i}
\sum_{i=1}^m (1-\lambda_i)\cr
&=& \frac{1}{m}{\rm Tr}((a_{j+1}I - A_j)^{-1}){\rm Tr}(I-A_j)\cr
&=& \frac{1}{m} \Phi^{a_{j+1}}(A_j){\rm Tr}(I - A_j)\cr
&\leq& \frac{1}{m} \Phi^{a_j}(A_j){\rm Tr}(I-A_j)\cr
&\leq& \frac{1}{m} \Phi^{a_0}(A_0){\rm Tr}(I-A_j)\cr
&=& \frac{1}{\sqrt{\epsilon}}{\rm Tr}(I - A_j).
\end{eqnarray*}

Next, $a_{j+1} - a_j = \frac{1}{1-\sqrt{\epsilon}}\cdot\frac{1}{n}$, so we can
estimate
\begin{eqnarray*}
\Phi^{a_j}(A_j) - \Phi^{a_{j+1}}(A_j)
&=& {\rm Tr}( (a_jI - A_j)^{-1} - (a_{j+1}I - A_j)^{-1})\cr
&=& \frac{1}{1-\sqrt{\epsilon}}\cdot\frac{1}{n}
{\rm Tr}( (a_jI - A_j)^{-1}(a_{j+1}I - A_j)^{-1})\cr
&>& \frac{1}{1-\sqrt{\epsilon}}\cdot\frac{1}{n}
{\rm Tr}((a_{j+1}I - A_j)^{-2})
\end{eqnarray*}
since each of the eigenvalues
$\frac{1}{a_j - \lambda_i}\frac{1}{a_{j+1} - \lambda_i}$ of the operator
$(a_jI - A_j)^{-1}(a_{j+1}I - A_j)^{-1}$ is greater than the corresponding
eigenvalue $\frac{1}{(a_{j+1} - \lambda_i)^2}$ of the operator
$(a_{j+1}I - A_j)^{-2}$. Combining this with Lemma \ref{lem2} yields
\begin{eqnarray*}
{\rm Tr}((a_{j+1}I - A_j)^{-2}(I - A_j))
&\leq& \frac{1}{m}{\rm Tr}((a_{j+1}I - A_j)^{-2}){\rm Tr}(I - A_j)\cr
&<& \frac{1}{\epsilon}(1 - \sqrt{\epsilon})
(\Phi^{a_j}(A_j) - \Phi^{a_{j+1}}(A_j)){\rm Tr}(I-A_j).
\end{eqnarray*}
Thus
\[\frac{{\rm Tr}((a_{j+1}I - A_j)^{-2}(I-A_j))}{\Phi^{a_j}(A_j) -
\Phi^{a_{j+1}}(A_j)}
\leq \frac{1}{\epsilon}(1 - \sqrt{\epsilon}){\rm Tr}(I-A_j).\]

Now let $S' \subseteq \{1, \ldots, n\}$ be the set of indices which
have not yet been used. Observe that
$\langle Au,u\rangle = {\rm Tr}(A(uu^*))$ and that
$\sum_{i \in S'} u_iu_i^* = I - \sum_{d = 1}^j u_{i_d}u_{i_d}^*
= I - A_j$. Thus
\begin{eqnarray*}
&&\sum_{i \in S'} \left(\frac{\langle (a_{j+1}I -
A_j)^{-2}u_i,u_i\rangle}{\Phi^{a_j}(A_j) - \Phi^{a_{j+1}}(A_j)}
+ \langle (a_{j+1}I - A_j)^{-1}u_i,u_i\rangle\right)\cr
&&= \frac{{\rm Tr}((a_{j+1}I - A_j)^{-2}(I - A_j))}{\Phi^{a_j}(A_j) -
\Phi^{a_{j+1}}(A_j)}
+ {\rm Tr}((a_{j+1}I - A_j)^{-1}(I - A_j))\cr
&&\leq \frac{1}{\epsilon}(1 - \sqrt{\epsilon}){\rm Tr}(I-A_j) +
\frac{1}{\sqrt{\epsilon}}{\rm Tr}(I - A_j)\cr
&&= \frac{1}{\epsilon}{\rm Tr}(I - A_j).
\end{eqnarray*}
But
\[\frac{1}{\epsilon}{\rm Tr}(I-A_j) = \frac{1}{\epsilon}(m - {\rm Tr}(A_j))
= n - j\]
is exactly the number of elements of $S'$. So there must exist some $i \in S'$
for which
\[\frac{\langle (a_{j+1}I - A_j)^{-2}u_i,u_i\rangle}{\Phi^{a_j}(A_j)
- \Phi^{a_{j+1}}(A_j)}
+ \langle (a_{j+1}I - A_j)^{-1}u_i,u_i\rangle
\leq 1.\]
Therefore, by Lemma \ref{lem1}, choosing $u_{i_{j+1}} = u_i$ allows the
inductive construction to proceed.
\end{proof}

Theorem \ref{onesided} follows by taking $m = {\rm dim}(F)$,
$\epsilon = m/|G|$, and $u_i = Qe_i$ for $1 \leq i \leq n = |G|$,
and identifying $F$ with ${\bf C}^m$. Letting $P$ be the orthogonal
projection onto ${\rm span}\{e_i: i \in S\}$, we then have
$\|PQ\|^2 = \|QPQ\| = \|\sum_{i \in S} u_iu_i^*\|$.

\end{document}